\documentclass{amsart}
\usepackage[matrix,cmtip,arrow,curve]{xy}
\usepackage{tikz-cd}
\usepackage[unicode]{hyperref}
\usepackage{amssymb}
\usepackage{stmaryrd}
\usepackage{footnote}
\usepackage{amsmath}
\usepackage{amsthm}
\usepackage{amscd} 
\usepackage{enumitem}

\usepackage{tikz}
\usetikzlibrary{arrows,positioning}
\usepackage{tikz-cd}
\usetikzlibrary{decorations.pathmorphing}

\theoremstyle{plain}
\newtheorem{thm}{Theorem}[section]
\newtheorem{theorem}[thm]{Theorem}

\newtheorem{proposition}[thm]{Proposition}
\newtheorem{lemma}[thm]{Lemma}
\newtheorem{corollary}[thm]{Corollary}
\newtheorem*{mainthm}{Main Theorem}
\theoremstyle{definition}
\newtheorem{definition}[thm]{Definition}
\newtheorem{example}[thm]{Example}

\theoremstyle{remark}
\newtheorem{remark}[thm]{Remark}

\newcommand{\sCat}{\ensuremath{\mathbf{sCat}}}
\newcommand{\sSet}{\mathbf{sSet}}
\newcommand{\sProp}{\ensuremath{\mathbf{sProp}}}
\newcommand{\Cat}{\ensuremath{\mathbf{Cat}}}
\newcommand{\sProperad}{\ensuremath{\mathbf{sProperad}}}

\newcommand{\inv}{{-1}}
\newcommand{\Set}{\mathbf{Set}}
\newcommand{\Hom}{\operatorname{Hom}}
\DeclareMathOperator*{\colim}{\operatorname{colim}}
\newcommand{\ua}{\underline a}
\newcommand{\ub}{\underline b}
\newcommand{\ba}{\binom{\ub}{\ua}}
\newcommand{\fC}{\mathfrak{C}}
\newcommand{\uc}{\underline{c}}
\newcommand{\ud}{\underline{d}}
\newcommand{\dc}{\binom{\ud}{\uc}}
\newcommand{\fdc}{\binom{f\ud}{f\uc}}
\newcommand{\profilev}{(\inp(v);\out(v))}
\newcommand{\dch}{(\uc;\ud)}
\DeclareMathOperator{\vertex}{Vt}
\DeclareMathOperator{\inp}{in}
\DeclareMathOperator{\out}{out}
\DeclareMathOperator{\Ob}{Ob}
\newcommand{\id}{\operatorname{id}}
\newcommand{\col}{\operatorname{Col}}

\newcommand{\ord}{\operatorname{ord}}

\newcommand{\calB}{\mathcal B}
\newcommand{\calC}{\mathcal C}
\newcommand{\calD}{\mathcal D}
\newcommand{\calG}{\mathcal G}
\newcommand{\calH}{\mathcal H}
\newcommand{\calI}{\mathcal I}
\newcommand{\calP}{\mathcal P}
\newcommand{\calQ}{\mathcal Q}
\newcommand{\calR}{\mathcal R}
\newcommand{\calW}{\mathcal W}

\newcommand{\Gr}{\mathtt{Gr}}

\renewcommand{\profilev}{\binom{\out(v)}{\inp(v)}}

\newcommand{\fD}{{\mathfrak D}}
\newcommand{\fdch}{(f\uc;f\ud)}
\newcommand{\fba}{\binom{f\ub}{f\ua}}
\newcommand{\calE}{\mathcal E}
\newcommand{\calF}{\mathcal F}

\newcommand{\calL}{\mathcal L}
\newcommand{\calK}{\mathcal K}

\newcommand{\wm}{\mathcal W_{\bM}}
\newcommand{\wn}{\mathcal W_{\bN}}

\newcommand{\rlp}[1]{{#1}^\boxslash}
\newcommand{\llp}[1]{{}^\boxslash\!{#1}}

\newcommand{\bM}{\mathbf M}
\newcommand{\bN}{\mathbf N}
\newcommand{\bP}{\mathbf P}

\newcommand{\sDioperad}{\ensuremath{\mathbf{sDioperad}}}
\newcommand{\Iso}{\operatorname{Iso}}

\newcommand{\dcprime}{(\uc';\ud')}

\newcommand{\fE}{\mathfrak{E}}

\title{A simplicial model for infinity properads}

\author[P. Hackney]{Philip Hackney}
\address{Institut f\"ur Mathematik \\ Universit\"at Osnabr\"uck \\ Osnabr\"uck, Germany}
\curraddr{Department of Mathematics\\ Macquarie University\\ NSW 2109 \\ Australia}
\email{philip@phck.net}

\author[M. Robertson]{Marcy Robertson}
\address{School of Mathematics and Statistics \\ The University of Melbourne \\ Melbourne, Victoria, Australia}\email{marcy.robertson@unimelb.edu.au}

\author[D. Yau]{Donald Yau}
\address{Department of Mathematics\\
The Ohio State University at Newark\\
Newark, OH \\ USA}
\email{dyau@math.ohio-state.edu}

\begin{document}

\begin{abstract}
	We show how the model structure on the category of simplicially-enriched (colored) props induces a model structure on the category of simplicially-enriched (colored) properads.
	A similar result holds for dioperads.
\end{abstract}

\subjclass[2010]{55P48, 18D20, 55U35, 18D50, 19D23}
\keywords{Properads, dioperads, infinity properads, model categories, cofibrantly generated model categories, simplicial categories}

\maketitle

This short note is an important component in an ongoing project to understand `up-to-homotopy' prop(erad)s.
Props, properads, and dioperads are devices like operads, but which are capable of controlling bialgebraic structures.
The notion of prop originated in the work of Adams and MacLane \cite{catalg}, while properads were introduced much later by Vallette \cite{vallette}.
Some of the best known properads include those that govern Lie bialgebras and Frobenius algebras (see, for example, \cite{MR2397628}).

Dioperads, like properads, are smaller versions of props defined by pasting schemes of graphs which are simply connected. 
A dioperad (which first appear in the thesis of Gan \cite{gan}; see also \cite{MR0373846, MR2414322}) describes an algebraic structure that has a multiplication and a comultiplication with relations which can be represented by simply connected graphs.  
As an illustrative example, one should note that a dioperad can describe the structure of a Lie bialgebra but not a bialgebra. 

In \cite{hrybook} we construct a combinatorial model for objects like properads, but where the properadic structure only holds up to coherent higher homotopy.
There, we present such `infinity properads' as objects of the presheaf category $\Set^{\Gamma^{op}}$ satisfying inner-horn filling conditions, where $\Gamma$ is a certain category of graphs.
The category $\Gamma$ is an extension of both the simplicial category $\Delta$ and the Moerdijk-Weiss dendroidal category $\Omega$ \cite{mw}, and our definition of infinity properads is analogous to that of quasi-categories \cite{joyal} (or infinity categories \cite{htt}) and dendroidal inner Kan complexes \cite{mw2}.
In a future paper we will prove the existence of a Quillen model structure on the category of graphical sets $\Set^{\Gamma^{op}}$ so that the fibrant objects are precisely the infinity properads; antecedents to this structure  are the Joyal model structure on simplicial sets $\Set^{\Delta^{op}}$ \cite{joyal,htt} and the Cisinski-Moerdijk model structure on dendroidal sets $\Set^{\Omega^{op}}$ \cite{cm-ho}.

In the present work, we study (small) \emph{simplicially-enriched properads}, which we expect to be the rigid model for infinity properads, much as simplicially-enriched categories \cite{bergner} give a model for infinity(-one) categories and simplicially-enriched operads give a model for infinity operads \cite{cm-simpop}.
Namely, in \cite{hry3} we will present a functor, called the `homotopy coherent nerve'
\[
	N_{hc} : \sProperad \to \Set^{\Gamma^{op}}
\]
which we anticipate, in analogy with the corresponding result in the categorical setting \cite{joyal_rigid,htt}, will be the right adjoint in a Quillen-equivalence of model categories.\footnote{This would also provide an alternate proof of the equivalence between the category of simplicial operads and that of dendroidal sets, which appears in \cite{cm-ho,cm-simpop,cm-ds}.}
For such a theorem to even be stated, we of course require a model structure on $\sProperad$, the category of small simplicially-enriched properads (henceforth called `simplicial properads').

Given a simplicial prop, properad, or dioperad $\calP$, we can look at its underlying simplicial category by discarding all $\calP \dc$ with $|\uc| \neq 1 \neq |\ud|$.
Further, given a simplicial category $\calC$, we can get a discrete category of components $\pi_0\calC$ by setting $\Ob \pi_0 \calC = \Ob \calC$ and $(\pi_0 \calC)(a, b) = \pi_0 (\calC(a,b))$.
For concision, we will just write $\pi_0$ for any of the composites  
\[ \begin{tikzcd}[row sep=small]
\sProp \dar[swap]{\text{forget}} \arrow[dr, bend right=10] \\
\sProperad \dar[swap]{\text{forget}}  \rar & \sCat \rar{\pi_0} & \Cat \\
\sDioperad \arrow[ur, bend left=10]
\end{tikzcd} \]
from one of the categories on the left into $\Cat$.

\begingroup
\def\thethm{A}
\addtocounter{thm}{-1}
\begin{definition}\label{WEandfibrations} Let $f:\calP \to \calQ$ be a morphism of simplicial props, properads, or dioperads.
We say that $f$ is a \emph{weak equivalence} if
\begin{enumerate}[label={(W{\arabic*})},ref={W\arabic*}]
\item\label{W1} for each input-output profile $\ba$ in $\col(\calP)$ (that is, pair of lists of colors of $\calP$) the morphism 
\[
	f:\calP\ba\longrightarrow\calQ\fba
\] 
is a weak homotopy equivalence of simplicial sets; and
\item\label{W2} the functor $\pi_{0}f:\pi_{0}\calP\rightarrow\pi_{0}\calQ$ is an equivalence of categories. 
\end{enumerate}
We say that the morphism $f$ is a \emph{fibration} if
\begin{enumerate}[label={(F{\arabic*})},ref={F\arabic*}]
\item\label{F1} for each input-output profile $\ba$ in $\col(\calP)$ the morphism 
\[
	f:\calP\ba\longrightarrow\calQ\fba 
\] 
is a Kan fibration of simplicial sets; and
\item\label{F2} the functor $\pi_{0}f:\pi_{0}\calP\rightarrow\pi_{0}\calQ$ is an isofibration.\footnote{A functor $p: \calE \to \calB$ in $\Cat$ is called an \emph{isofibration} if 
for each isomorphism $h: p(e) \to b$ in $\calB$, there exists an isomorphism $g: e \to e'$ in $\calE$ with $p(g) = h$.}
\end{enumerate}
\end{definition}
\endgroup

The main thereom of \cite{bergner} states that $\sCat$ admits a model structure\footnote{This model structure is cofibrantly generated. Sets of generating (acyclic) cofibrations are recalled in Definition \ref{generatingcof}.} so that a map $f: \calC \to \calD$ is a weak equivalence (respectively, fibration) if and only if it is locally one (that is, $f: \calC(a,b) \to \calD(fa, fb)$ is one for all $a,b\in \Ob(\calC)$) and if $\pi_0f$ is an equivalence of categories (respectively, isofibration).

\begin{mainthm}
	The category of simplicial properads and the category of simplicial dioperads admit model structures with the weak equivalences and fibrations from Definition \ref{WEandfibrations}.
\end{mainthm}

We should first point out that these model structures cannot be lifted from the model structure on simplicial operads \cite{cm-simpop}, as the conditions \eqref{W1}, \eqref{F1} in Definition \ref{WEandfibrations} would only be relevant when $|\ub| = 1$. 

It is possible to prove the main theorem (at least in the case of simplicial properads) by imitating the proofs in \cite{hackneyrobertson2}.
This has the benefit that it requires no new ideas, but this approach is both technically difficult and tedious.
The approach we take in this paper rests on Proposition \ref{squeezy theorem}, which we find novel and interesting in its own right. 
Were are aware of only two precursors in the literature. The first is the way that Hovey restricts the model structure on all topological spaces to the coreflective subcategory of Kelley spaces \cite[2.4.23]{hovey},
while the second is Corollary \ref{corollary thing}, which was originally due to Intermont and Johnson in their study of ex-spaces \cite[Lemma 8.8]{MR1888677}.
Proposition \ref{squeezy theorem} allows us to \emph{apply} the results\footnote{See also later results of Caviglia \cite{caviglia}, who extended these model structures to enriching categories other than $\sSet$.} of the first two authors \cite{hackneyrobertson2} to obtain the desired model structure on $\sProperad$.

In the next section, we recall a few ideas from the theory of Quillen model categories.
Proposition \ref{squeezy theorem} seems to be new, and is a primary technical tool in the proof of the main theorem of the paper.
In section \ref{section graphs}, we will recall some notation and definitions about graphs and generalized props, most of which can be found in \cite{yj}. 
Sections \ref{section local} and \ref{section adj} together show that most of the hypotheses of Proposition \ref{squeezy theorem} hold.
Section \ref{section local} is about the structure of local equivalences in our categories of generalized props, and does not deal with any adjunctions.
Section \ref{section adj} illustrates quite nicely why the main theorem of this paper is \emph{not} formal -- we really depend on some internal structures in our objects of interest.
Finally, in section \ref{section model}, we actually apply Proposition \ref{squeezy theorem} to prove the main theorem.

\subsection*{Acknowledgments} 
The authors would like to thank Alexander Berglund, Martin Markl, Irakli Patchkoria, and Emily Riehl.

\section{Cofibrantly generated model categories}\label{section cgmc}

Let $\bM$ be a category and $\bM^{[1]}$ its category of arrows. 
We now borrow some notation from \cite{riehl}.
If 
$i : A \to B, f: X \to Y$ are morphisms of $\bM$ (that is, objects of $\bM^{[1]}$), we write $i\boxslash f$ if the map
\begin{align*}
	\hom_{\bM} (B, X) &\to \hom_{\bM^{[1]}} (i, f)  \\
	g & \mapsto \left( \begin{tikzcd}[ampersand replacement=\&]
		A \dar{i} \rar{g\circ i} \& X \dar{f} \\
		B \rar{f \circ g} \& Y
	\end{tikzcd}\right) 
\end{align*}
is surjective. 
In other words, $i\boxslash f$ if, for every commutative diagram
\[ \begin{tikzcd}
	A \rar \dar{i} & X \dar{f} \\
	B \rar \arrow[dotted]{ur} & Y,
\end{tikzcd} \]
a lift $B\to X$ exists, making both triangles commute. 
If $\calK$ is a class of maps in $\bM$, we write $\rlp{\calK}$ for the collection of morphisms which have the right lifting property with respect to $\calK$; that is, $\rlp{\calK}$ is the collection of $h$ satisfying $k\boxslash h$ for all $k\in \calK$.
Similarly, we write $\llp{\calK}$ for the collection of all $h$ so that $h\boxslash k$ for all $k\in \calK$.

Suppose that $\mathcal K$ is a class of maps in some category $\bM$.
A map $f$ is a $\calK$-cell complex, that is, $f\in \calK\text{-cell}$, if it is a transfinite composition of pushouts of elements of $\calK$.

\begin{lemma}\label{Lcell}
If $L \colon \bM \to \bN$ is a functor which
preserves colimits and $\calK$ is a class of maps in $\bM$, then $L(\calK\text{-cell}) \subset (L\calK)\text{-cell}$. \qed
\end{lemma}

If $\calK$ is a class of maps in $\bN$ and $F : \bM \to \bN$ is any functor, we write $F^\inv(\calK)$ for the class consisting of all maps $f$ so that $F(f) \in \calK$.

\begin{lemma}\label{rlp adjunction}
	Let 
\[
	L \colon \bM \rightleftarrows \bN \colon U
\]
be an adjoint pair of functors.
If $\calK$ is a class of maps in $\bM$, then 
\[
	U^{-1}(\rlp{\calK}) = \rlp{(L\calK)}.
\]
\end{lemma}
\begin{proof}
	We have $ f : X \to Y \in U^{-1}(\rlp{\calK})$ if and only if $U(f) \in \rlp{\calK}$.
	This is equivalent to the map
	\begin{align*}
		\hom_{\bM} (B, UX) &\to \hom_{\bM^{[1]}} (k, Uf)  \\
		g & \mapsto \left( \begin{tikzcd}[ampersand replacement=\&]
			A \dar{k} \rar{g\circ k} \& UX \dar{Uf} \\
			B \rar{Uf \circ g} \& UY
		\end{tikzcd}\right) 
	\end{align*}
	being surjective for all $k\in \calK$.
	By adjointness (which extends to the level of arrow categories), this is equivalent to surjectivity of 
	\[ \hom_{\bN} (LB, X) \to \hom_{\bN^{[1]}} (Lk, f) \]
	for all $k\in \calK$, i.e., $f\in \rlp{(L\calK)}$.
\end{proof}

Let $\bM$ be a cocomplete category and $A\in \bM$ an object. 
We say that $A$ is \emph{finite} if for every sequence
$X_0 \to X_1 \to \cdots \to X_n \to \cdots$ 
indexed by the natural numbers $\mathbb N$, the map
\[
	\colim_i \bM(A,X_i) \to \bM(A, \colim_i X_i)
\]
is an isomorphism.
There is a more general version of this, where one can speak of an object $A$ being small relative to a class of maps $\calK$ in $\bM$ (see \cite[2.1.3]{hovey}), but in our applications we only deal with finite objects, which are small relative to any class of maps in $\bM$.

\begin{definition}
	A model category $\bM$ is \emph{cofibrantly generated} if there are sets $I$ and $J$ of maps such that
	\begin{itemize}
	\item The domains of $I$ are small relative to $I\text{-cell}$;
	\item The domains of $J$ are small relative to $J\text{-cell}$;	
	\item The class of fibrations is $\rlp{J}$; and
	\item The class of acyclic fibrations is $\rlp{I}$.
	\end{itemize}
\end{definition}

Such a cofibrantly generated model category has $\llp{(\rlp{I})}$ as its class of cofibrations and $\llp{(\rlp{J})}$ as its class of acyclic cofibrations.

Recall the following recognition theorem \cite[2.1.19]{hovey} for cofibrantly generated model categories.

\begin{theorem}\label{Kan theorem}
Let $\bM$ be a bicomplete category, $\calW$ a subcategory of $\bM$, and $I,J$ sets of maps of $\bM$. Then there is a cofibrantly generated model category structure on $\bM$ with $I$ as the set of generating cofibrations, $J$ as the set of generating acyclic cofibrations, and $\calW$ as the subcategory of weak equivalences if and only if the following are satisfied:
\begin{enumerate}[label=\emph{({\Roman*})},ref={\thethm.\Roman*}]
	\item The subcategory $\calW$ has the two out of three property and is closed under retracts.  \label{KT1}
	\item The domains of $I$ are small relative to $I\text{-cell}$. \label{KT2}
	\item The domains of $J$ are small relative to $J\text{-cell}$. \label{KT3}
	\item $J\text{-cell} \subset \calW \cap \llp{(\rlp{I})}$. \label{KT4}
	\item $\rlp{I} \subset \calW \cap \rlp{J}$. \label{KT5}
	\item Either $\calW \cap \llp{(\rlp{I})} \subset \llp{(\rlp{J})}$ or $\calW \cap \rlp{J} \subset \rlp{I}$. \label{KT6}
\end{enumerate}
\qed
\end{theorem}

Notice that both parts of \eqref{KT6} hold simultaneously in any cofibrantly generated model category $\bM$.

Given a pair of adjunctions
\[ \begin{tikzcd}
\bM \arrow[r, shift left, "F_1"] & 
\bN \arrow[r, shift left, "F_2"] \arrow[l, shift left, "U_1"] &
\bP \arrow[l, shift left, "U_2"] 
\end{tikzcd} \]
we shall call the adjunction $(F_2, U_2)$ an \emph{adjunction over $\bM$}.

\begin{proposition}\label{squeezy theorem}
Let $\bM, \bN, \bP$ be bicomplete categories and let $\calL \subset \bN$ be a class of maps and $I, J \subset \bN$ be sets of maps.
Suppose further that there is an adjunction $(F_2,U_2)$ over $\bM$:
\[ \begin{tikzcd}
\bM \arrow[r, shift left, "F_1"] \arrow[rr, bend left=50, swap,  "F_0"]& 
\bN \arrow[r, shift left, "F_2"] \arrow[l, shift left, "U_1"] &
\bP \arrow[l, shift left, "U_2"] \arrow[ll, bend left=50, swap,  "U_0"]
\end{tikzcd} \]
Assume that the following hold.
\begin{enumerate}[label=\emph{({\Alph*})},ref={\Alph*}]
	\item
	$\bM$ admits the structure of a cofibrantly-generated model category with weak equivalences $\wm$ and generating (acyclic) cofibrations $I_0$ (resp. $J_0$).\label{ST1} 
	\item
	$\bP$ admits the structure of a cofibrantly-generated model category with weak equivalences $\calW_{\bP}
	$ and generating (acyclic) cofibrations $F_0I_0\cup F_2I$ (resp. $F_0J_0\cup F_2J$).\label{ST2}
	\item The subcategory $\wn = (U_1^\inv\wm) \cap \calL$ has the two out of three property and is closed under \label{ST5}retracts.\footnote{Note that $U_1^\inv\wm$ automatically satisfies the two out of three property and is closed under retracts; thus it is sufficient (but not necessary) to show that $\calL$ satisfies two out of three and is closed under retracts. Indeed, in our applications of this theorem, $\calL$ will not satisfy two out of three.}
	\item The domains of $F_1 I_0 \cup I$ are small relative to $(F_1 I_0 \cup I)\text{-cell}$.\label{ST3}
	\item The domains of $F_1 J_0 \cup J$ are small relative to $(F_1 J_0 \cup J)\text{-cell}$.\label{ST4}
	\item
	$\rlp I = \calL \cap \rlp J$.\label{ST6}
	\item
	$F_2^{-1} \left( \calW_{\bP} \right) \subset \calW_{\bN} 
	$.\label{ST7}
\end{enumerate}

Then $\bN$ admits the structure of a cofibrantly generated model category, with weak equivalences $\wn = (U_1^\inv\wm) \cap \calL$, generating cofibrations $F_1I_0 \cup I$, and generating acyclic cofibrations $F_1J_0 \cup J$.
\end{proposition}
\begin{proof} We will apply Theorem \ref{Kan theorem}.

We can simultaneously show that \eqref{KT5} and \eqref{KT6} hold.
We have
\begin{equation}\label{fizero}
\begin{aligned}
	\rlp{(F_1I_0)} &= U_1^\inv(\rlp{I_0}) \\ &= U_1^\inv(\wm \cap \rlp{J_0}) & \eqref{ST1}\\
	 &= U_1^\inv(\wm) \cap U_1^\inv(\rlp{J_0}) \\ &= U_1^\inv(\wm) \cap \rlp{(F_1 J_0)} & \text{Lemma \ref{rlp adjunction}}
\end{aligned} \end{equation}
and thus
\begin{equation} \label{I equality} \begin{aligned}
	\rlp{(F_1 I_0 \cup I)} & = 
	\rlp{(F_1I_0)} \cap \rlp I \\ &=
	U_1^\inv(\wm) \cap \rlp{(F_1 J_0)} \cap \calL \cap \rlp J & & \eqref{fizero}, \eqref{ST6}  \\ &= 
	U_1^\inv(\wm) \cap \calL \cap \rlp{(F_1 J_0)} \cap \rlp J \\ &=
	\wn \cap \rlp{(F_1 J_0 \cup J)}
\end{aligned} \end{equation}

We now turn to \eqref{KT4}.
For conciseness, write $I' = F_1 I_0 \cup I$, $J' = F_1 J_0 \cup J$; then $F_2I' = F_0I_0 \cup F_2 I$ and likewise for $J$.
Suppose that $f \in J'\text{-cell}$. By Lemma \ref{Lcell}, \[F_2(f) \in (F_2J')\text{-cell} \overset{\eqref{ST2}}{\subset} \calW_{\bP} \cap \llp{(\rlp{(F_2I')})}.\]
Thus, by \eqref{ST7}, $f\in \wn$.
Since $\rlp{I'} \overset{\eqref{I equality}}\subset \rlp{J'}$, we have $\llp{(\rlp{I'})} \supset \llp{(\rlp{J'})}$ and of course $J'\text{-cell} \subset \llp{(\rlp{J'})}$. 
This shows $J'\text{-cell} \subset \llp{(\rlp{I'})}$, hence
\[ (F_1 J_0 \cup J)\text{-cell} \subset \wn \cap \llp{(\rlp{(F_1 I_0 \cup I)})}.\]

We have now established \eqref{KT4}--\eqref{KT6} of Theorem \ref{Kan theorem}; conditions \eqref{KT1}--\eqref{KT3} were assumed (as \ref{ST5}, \ref{ST3}, \ref{ST4}) to be true.
Thus $\bN$ admits the desired cofibrantly generated model structure.
\end{proof}

\begin{remark}
Notice by \cite[2.1.20]{hovey}, that all adjunctions in this theorem statement are Quillen adjunctions.
\end{remark}

The following is a baby version of the above proposition, which we include here for completeness rather than for any further use in this paper.
It originally appeared as \cite[Lemma 8.8]{MR1888677}, and we thank M.\,Johnson for alerting us to this fact.

\begin{corollary}\label{corollary thing}
	Suppose that $i : \bN \subseteq \bP$ is a coreflective full subcategory with $\bN$ and $\bP$ bicomplete.
	Assume that there are sets $I, J \subset \bN$ so that $\bP$ admits the structure of a cofibrantly generated model category with $iI$ (resp. $iJ$) as its set of generating (acyclic) cofibrations, weak equivalences $\calW$, and fibrations $\calF$.
	Then $\bN$ also admits the structure of a cofibrantly generated model category, with $I$ (resp. $J$) the generating (acyclic) cofibrations, weak equivalences $\calW \cap \bN$, and fibrations $\calF \cap \bN$.
\end{corollary}
A prototypical example of this situation is the inclusion of the category of Kelley spaces into the category of all topological spaces; in this setting, this corollary essentially appears in \cite[2.4.22 -- 2.4.23]{hovey}.
Other situations are those in which $i: \bN \to \bP$ is fully faithful and the adjoint functor theorem applies, for instance when $\bN$, $\bP$ are locally-presentable and $i$ preserves all colimits.

\begin{proof}[Proof of Corollary \ref{corollary thing}]
	Let $u: \bP \to \bN$ be the right adjoint to $i: \bN \to \bP$, and let $\bM$ be the terminal category; there is an adjunction $F_1 \colon \bM \rightleftarrows \bN \colon U_1$ where $F_1$ sends the unique object of $\bM$ to the initial object $\varnothing$ of $\bN$.

	Apply Proposition \ref{squeezy theorem} to these adjunctions, with $I_0 = \varnothing = J_0$ and 
	\begin{equation}\label{calL definition}
		\calL = \calW \cap \bN = i^\inv(\calW).
	\end{equation}
	Hypothesis \eqref{ST1} is automatic and \eqref{ST2} holds by assumption.
	Notice that $U_1^\inv(\calW_{\bM}) = \bN$ and $U_0^\inv(\calW_{\bM}) = \bP$.
	Hypothesis \eqref{ST3} holds since the domains of elements of $I$ are already small relative to the larger\footnote{Lemma \ref{Lcell}} class $iI\text{-cell}$, hence to the class $I\text{-cell}$; in a similar way we obtain \eqref{ST4}.
	Hypothesis \eqref{ST5} follows from the corresponding properties of $\calW$.

	Since $i$ is fully faithful, $ui \cong \id_{\bN}$; thus if $\calK$ is any class of maps in $\bN$ 
	\begin{equation}\label{K in this instance}
	\rlp{(i\calK)} \cap \bN = i^\inv \left(\rlp{(i\calK)}\right) \overset{\ref{rlp adjunction}}= i^\inv \left( u^\inv \left( \rlp{\calK} \right)\right)= (ui)^\inv \rlp{\calK}= \rlp{\calK}.
	\end{equation}
	So
	\[
		\rlp{I} \overset{\eqref{K in this instance}}= \rlp{(iI)} \cap \bN = \calW \cap \rlp{(iJ)} \cap \bN = \calW \cap \bN \cap \rlp{(iJ)} \cap \bN \overset{\eqref{K in this instance}}=  \calL \cap \rlp{J}
	\]
	and we have established \eqref{ST6}. 
	Finally, \eqref{ST7} holds since 
	\begin{equation}\label{weak equivalence line} i^\inv(\calW) = \calW \cap \bN = (\calW \cap \bN) \cap \bN = \calL \cap U_1^\inv (\calW_{\bM}) = \calW_{\bN} \end{equation}
	since $U_1^\inv (\calW_{\bM}) = \bN$.

	Now we can apply Proposition \ref{squeezy theorem} to obtain a model structure on $\bN$. 
	The class of weak equivalences in $\bN$ is $\calW \cap \bN$ by \eqref{weak equivalence line}.
	We know that the fibrations for $\bN$ are
	\[
	\rlp{J} \overset{\eqref{K in this instance}}= \rlp{(iJ)} \cap \bN = \calF \cap \bN,
	\]
	which completes the proof.
\end{proof}

\section{Graph groupoids, pasting schemes, generalized props}\label{section graphs}

In this section we recall some concepts and examples from \cite{yj}, though we often use the same terminology for things that are much less general in the present paper.

Given a set $\fC$, a \emph{profile} $\uc = (c_1, \dots, c_n)$ is simply an ordered list of elements in $\fC$. A \emph{biprofile} is a pair of profiles, written alternately as
\[
	\binom{d_1, \dots, d_m}{c_1, \dots, c_n} = \dc = \dch = (c_1, \dots, c_n;d_1, \dots, d_m),
\]
where each $c_i$ and each $d_k$ are in $\fC$.

An \emph{$\fC$-colored graph} $G$ consists of
\begin{itemize}
\item a directed graph $G$ with half-edges which has no directed cycles, 
\item a coloring function $\xi$ from the set of edges of $G$ to $\fC$,
\item orderings on the inputs and outputs of the graph
\begin{align*}
	\ord_i : \{ 1, \dots, n \} &\overset{\cong}\to \inp G \\
	\ord_o : \{ 1, \dots, m \} &\overset{\cong}\to \out G,
\end{align*}
and
\item orderings on the inputs and outputs of each vertex $v\in \vertex(G)$
\begin{align*}
	\ord_i^v : \{ 1, \dots, n_v \} &\overset{\cong}\to \inp v \\
	\ord_o^v : \{ 1, \dots, m_v \} &\overset{\cong}\to \out v.
\end{align*}
\end{itemize}

\begin{example}
	Given a biprofile $\dch = (c_1, \dots, c_n; d_1, \dots, d_m)$ with $c_i, d_j \in \fC$, the \emph{standard corolla}
	$C_{\dch}$ is the graph with one vertex $v$, half-edges $1, \dots, n+m$ with
	\begin{align*}
		\ord_i =  \ord_i^v : \{ 1, \dots, n \} &\overset{=}\to \{1, \dots, n\} = \inp G = \inp v\\
		\ord_o = \ord_o^v: \{ 1, \dots, m \} &\overset{+n}\to \{n+1, \dots, n+m\} = \out G = \out v
	\end{align*}
	and 
	\[
		\xi(i) = \begin{cases}
			c_i & 1\leq i \leq n \\
			d_{i-n} & n+1 \leq i \leq n+m.
		\end{cases}
	\]
\end{example}

A \emph{strict isomorphism} between $\fC$-colored graphs preserves all structure, while a \emph{weak isomorphism} does not necessarily preserve the orderings.
The category of (wheel-free) graphs along with weak isomorphism gives us our first example of a \emph{graph groupoid}, which we denote by $\Gr^\uparrow$.
We are also interested in the following full subgroupoids of $\Gr^\uparrow$:
\begin{itemize}
	\item The subgroupoid $\Gr^\uparrow_\text{c}$ whose objects are the \emph{connected} graphs.
	\item The subgroupoid $\Gr^\uparrow_\text{di}$ whose objects are the \emph{simply connected} graphs.
\end{itemize}
If we fix a set of colors $\fC$, then we will write $\Gr^\uparrow(\fC) \subset \Gr^\uparrow$ (resp. $\Gr^\uparrow_\text{c}(\fC) \subset \Gr^\uparrow_\text{c}$ and $\Gr^\uparrow_\text{di}(\fC) \subset \Gr^\uparrow_\text{di}$) for the full subgroupoids of $\fC$-colored graphs. 
For a fixed biprofile $\dch = (c_1, \dots, c_n; d_1, \dots, d_m)$ with $c_i, d_j \in \fC$, there is a (non-full) subgroupoid $\Gr^\uparrow(\fC)\dc \subset \Gr^\uparrow(\fC) \subset \Gr^\uparrow$ with
\begin{itemize}
	\item objects those graphs with $\xi(\ord_i(s)) = c_s \in \fC$ and $\xi(\ord_o(t)) = d_t \in \fC$,
	\item morphisms the \emph{strict} isomorphisms.
\end{itemize}
The use of strict isomorphism guarantees preservation of the colors of the inputs and outputs.
There are analogously defined supgroupoids $\Gr^\uparrow_{\text{c}}(\fC)\dc \subset \Gr^\uparrow_{\text{c}}(\fC) \subset \Gr^\uparrow_{\text{c}}$ and $\Gr^\uparrow_{\text{di}}(\fC)\dc \subset \Gr^\uparrow_{\text{di}}(\fC) \subset \Gr^\uparrow_{\text{di}}$.

Each of $\Gr^\uparrow(\fC), \Gr^\uparrow_{\text{c}}(\fC),$ and $\Gr^\uparrow_{\text{di}}(\fC)$ is a $\fC$-colored \emph{pasting scheme} \cite[8.2]{yj} for any color set $\fC$, which essentially means that they are closed under the operation of \emph{graph substitution}. 

\begin{remark}\label{remark strict iso}
We will often work with strict isomorphism classes of graphs instead of the graphs themselves; this assumption guarantees that the above categories of graphs are \emph{small} categories.
We will also need this in section \ref{left adjoints} to ensure that the extension category has small hom sets.
\end{remark}

Let $\fC$ be a set of colors, and let $\Gr$ be one of $\Gr^\uparrow$, $\Gr^\uparrow_\text{c}$, or $\Gr^\uparrow_\text{di}$.
A simplicial $\Gr(\fC)$-prop consists of the data of
\begin{itemize}
	\item a family of simplicial sets
	\[
		\calP\dc \in \sSet,
	\]
	one for each biprofile $\dch$ in $\fC$; 
	\item unit elements $\id_c \in \calP \binom{c}{c}_0$; and
	\item composition maps 
	\[
		\gamma_G : \calP[G] \to \calP \binom{\xi \out G}{\xi \inp G}
	\]
	for each $\fC$-colored graph $G\in \Gr(\fC)$, where
\[ 
	\calP[G] = \prod_{v\in \vertex(G)} \calP \binom{\xi\out(v)}{\xi\inp(v)}
\]
is the graph $G$ with each vertex decorated by some element in the family.
\end{itemize}
These data should satisfy appropriate identity, associativity, and equivariance properties; we refer the reader to \cite{hrybook,jy} or \cite[10.39]{yj} for precise definitions.
We will frequently write $\col(\calP) = \fC$ for the set of colors of $\calP$.
A morphism $\calP \to \calQ$ from a $\Gr(\fC)$-prop to a $\Gr(\fD)$-prop consists of a set map $f:\fC \to \fD$ and a family of morphisms
\[
	\left\{ \calP \dc \to \calQ \fdc \right\}
\]
which are compatible with the composition maps and unit elements.
Let $\sProp^{\Gr}$ be the category, fibered over $\Set$, whose objects are simplicial $\Gr(\fC)$-props (as $\fC$ varies) and whose morphisms are as above. We shall call objects in $\sProp^{\Gr}$ simply `simplicial $\Gr$-props'.
\begin{itemize}
	\item Objects of $\sProp^{\Gr^\uparrow}$ are called simplicial \emph{props}, and we write $\sProp$ for this category.
	\item Objects of $\sProp^{\Gr^\uparrow_{\text{c}}} = \sProperad$ are called simplicial \emph{properads}.
	\item Objects of $\sProp^{\Gr^\uparrow_{\text{di}}} = \sDioperad$ are called simplicial \emph{dioperads}.
\end{itemize}

\section{Local equivalences and liftings}\label{section local}

Consider one of the graph groupoids $\Gr$ discussed above, and let $\bN = \sProp^{\Gr}$ be the category of simplicial props (for $\Gr = \Gr^\uparrow$), simplicial properads (for $\Gr = \Gr^\uparrow_{\text{c}}$), or simplicial dioperads (for $\Gr = \Gr^\uparrow_{\text{di}}$). 
Let
$\calL =  \calL_\bN \subset \bN$ denote the subcategory of \emph{local equivalences}, i.e. those maps $f : \calP \to \calQ$ so that for every biprofile $\dc$ of $\calP$, the map
\[
	f_{\dch} : \calP\dc \to \calQ \fdc
\]
is a weak equivalence in $\sSet$.

\begin{remark}\label{not twoofthree remark}
The subcategory $\calL$ does not satisfy the two out of three property. 
The functor $\Cat \hookrightarrow \sCat \to \bN$ allows us to regard $\Cat$ as a full subcategory of $\bN$. Then $\calL \cap \Cat$ is the class of full and faithful functors, which does not satisfy two out of three.
For another example, if $\varnothing$ is the initial object of $\bN$ (with $\col(\varnothing) = \varnothing)$), then, for any $\calR$ with $\calR\binom{\varnothing}{\varnothing} = \varnothing$, the map $\varnothing \to \calR$ is in $\calL$.
We then have that the triple
\[
	\varnothing \to \calP \overset{f}\to \calQ
\]
violates two out of three whenever $\calP\binom{\varnothing}{\varnothing} = \calQ\binom{\varnothing}{\varnothing} = \varnothing$ and $f\notin \calL$.
\end{remark}

On the other hand, $\calL$ is closed under composition and, if we have
\[
	\calP \overset{g}\to  \calQ \overset{f}\to \calR
\]
with $f$ and $f\circ g$ both in $\calL$, then $g\in \calL$ as well.

Furthermore, $\calL$ is closed under retracts.
Suppose we are given a retraction diagram
\[ \begin{tikzcd}
X_1  \rar{i_1} \arrow[rr,bend left, "\id"]  \dar{f} & Y_1 \rar{r_1}  \dar{g} & X_1 \dar{f}  \\
X_2  \rar{i_2} \arrow[rr,bend right, "\id"] & Y_2 \rar{r_2} & X_2
\end{tikzcd} \]
with $g\in \calL$. Then, for each biprofile $\dc$ in $X_1$ we have a retraction diagram
\[ \begin{tikzcd}
X_1\dc  \rar{i_1} \arrow[rr,bend left, "\id"]  \dar{f} & Y_1\binom{i_1\ud}{i_1\uc} \rar{r_1}  \dar{g}[swap]{\simeq} & X_1\dc \dar{f}  \\
X_2\fdc  \rar{i_2} \arrow[rr,bend right, "\id"] & Y_2\binom{gi_1\ud}{gi_1\uc}  \rar{r_2} & X_2\fdc
\end{tikzcd} \]
in $\sSet$. Since weak equivalences in $\sSet$ are closed under retracts,
\[
	f_{\dch} : X_1\dc \to X_2\fdc
\]
is a weak equivalence 
for all $\dc$, hence $f\in \calL$.

We now begin to work towards Theorem \ref{lemma twoofthree}, where we address the defect of $\calL$ noted in Remark \ref{not twoofthree remark}: if $g$ and $f\circ g$ are not just in $\calL$ but are also \emph{categorical equivalences} (that is, satisfy (W2) of Definition \ref{WEandfibrations}), then $f$ is also in $\calL$.
To establish this fact, we need to show that isomorphisms in $\pi_0\calP$ act on the components $\calP\dc$ via weak equivalences.
Write
\[
	U_1 : \bN \to \sCat
\]
for the forgetful functor, with $\Ob(U_1(\calP)) = \col(\calP)$ and $U_1(\calP)(a,b) = \calP\binom{b}{a}$.

\begin{lemma}
Let $\calP \in \bN$ and 
suppose that 
$a$ and $a'$ are vertices of $\calP\binom{b'}{b} = U_1\calP(b, b')$ which
represent the same class in $\pi_0 \calP(b,b')$. 
Consider the maps
	\begin{align*}
	(-\circ_i a), (-\circ_i a') \colon 
	\calP\binom{\ud}{c_1,\dots, c_{i-1},b',c_{i+1}, \dots, c_n} &\to 
	\calP\binom{\ud}{c_1,\dots, c_{i-1},b,c_{i+1}, \dots, c_n}  \\
	(a \circ_j -), (a' \circ_j -) \colon 
	\calP\binom{d_1,\dots, d_{j-1},b,d_{j+1}, \dots, d_m}{\uc} &\to 
	\calP\binom{d_1,\dots, d_{j-1},b',d_{j+1}, \dots, d_m}{\uc}.
\end{align*}
Then we have homotopies $|{-}\circ_i a | \simeq |{-}\circ_i a'|$ and $|a \circ_j {-}| \simeq | a' \circ_j {-}|$ after taking geometric realization.
\end{lemma}
\begin{proof}
We prove the statements for precomposition; postcomposition follows similarly.
Since geometric realization commutes with finite products, $|\calP|$ is a $\Gr$-prop enriched in topological spaces with structure maps given by
\[
	\gamma^{|\calP|}_G \colon \left|\calP\right|[G] = \left|\calP[G]\right| \overset{\left|\gamma^{\calP}_G\right|}{\longrightarrow} \left|\calP\dc\right|.
\]
Let $\alpha$ be a path from $a$ to $a'$ in $| \calP\binom{b'}{b} |$, that is, a continuous map $I \to | \calP\binom{b'}{b} |$ with $\alpha(0) = a$ and $\alpha(1) = a'$.
The map
\[
	H : \left| \calP\binom{\ud}{c_1,\dots, c_{i-1},b',c_{i+1}, \dots, c_n} \right| \times I  \to \left| \calP\binom{\ud}{c_1,\dots, c_{i-1},b,c_{i+1}, \dots, c_n} \right|
\]
with $H(x,t) = x \circ_i \alpha(t)$ is a homotopy from $|{-}\circ_i a|$ to $|{-}\circ_i a'|$.
\end{proof}

\begin{lemma}
Let $\calP \in \bN$, and
suppose that $\alpha : b \to b'$ is an isomorphism in $\pi_0 U_1 \calP$.
If $a \in \calP\binom{b'}{b}_0 = U_1\calP(b, b')_0$ is a representative for $\alpha$, then the maps
\begin{align*}
	{-}\circ_i a \colon 
	\calP\binom{\ud}{c_1,\dots, c_{i-1},b',c_{i+1}, \dots, c_n} &\to 
	\calP\binom{\ud}{c_1,\dots, c_{i-1},b,c_{i+1}, \dots, c_n}  \\
	a \circ_j {-} \colon 
	\calP\binom{d_1,\dots, d_{j-1},b,d_{j+1}, \dots, d_m}{\uc} &\to 
	\calP\binom{d_1,\dots, d_{j-1},b',d_{j+1}, \dots, d_m}{\uc}
\end{align*}
are weak equivalences in $\sSet$.
\end{lemma}
\begin{proof}
We only prove the first statement; the second is similar.
Write 
\[X = \calP\binom{\ud}{c_1,\dots, c_{i-1},b,c_{i+1}, \dots, c_n}  \text{ and } X' = \calP\binom{\ud}{c_1,\dots, c_{i-1},b',c_{i+1}, \dots, c_n}. \] 
Let $a' \in U_1\calP(b',b)_0$ be a vertex so that $\alpha^{-1} = [a']$.
Then $[a \circ a'] = [\id_{b'}] \in \pi_0\calP\binom{b'}{b'}$ and $[a' \circ a] = [\id_{b}] \in \pi_0\calP\binom{b}{b}$.
By the previous lemma, $|{-}\circ_i (a'\circ a)| \simeq |{-}\circ_i {\id_{b}}| = \id_{|X|}$ and $|{-}\circ_i (a\circ a')|\simeq \id_{|X'|}$.
But
\begin{align*} \id_{|X|} \simeq |{-}\circ_i (a'\circ a)| &= |{-} \circ_i a'| \circ |{-}\circ_i a| \\
\id_{|X'|} \simeq |{-}\circ_i (a\circ a')| &= |{-} \circ_i a| \circ |{-}\circ_i a'| \end{align*}
so $|{-} \circ_i a|$ and $|{-}\circ_i a'|$ are homotopy inverses for each other. Thus $({-} \circ_i a)$ is a weak equivalence in $\sSet$.
\end{proof}

\begin{corollary}\label{change of objects}
Let $\dc$, $\binom{\ud'}{\uc'}$ be biprofiles, and suppose we have isomorphisms
	\begin{align*}
	\alpha_i &\in  \pi_0U_1\calP(c_i', c_i) \\
	\beta_j &\in \pi_0U_1\calP(d_j, d_j').
\end{align*}
By choosing representatives $a_i \in \alpha_i$ and $b_j\in \beta_j$, we have a map
\[
	(\underline a; \underline b)
	\colon  \calP\dc \to \calP\binom{\ud'}{\uc'}; 
\]
this map is a weak equivalence in $\sSet$. \qed
\end{corollary}

The following verifies \eqref{ST5} of Theorem \ref{squeezy theorem} in the proof of the main theorem (using $\Gr = \Gr^\uparrow_\text{c}$ in \ref{properadstheorem} and $\Gr = \Gr^\uparrow_\text{di}$ in \ref{dioperadstheorem}).

\begin{theorem}\label{lemma twoofthree}
Let $\bN = \sProp^{\Gr}$ be the category of simplicial props, simplicial properads, or simplicial dioperads.
	Suppose that we have two morphisms
	\[
	\calP \overset{g}\to  \calQ \overset{f}\to \calR
    \] of $\bN$. 
    If $g$ and $f\circ g$ are in 
    \[ \wn = (U_1^\inv\calW_{\sCat}) \cap \calL, \]
    then so is $f$.
    Consequently, $\wn$ satisfies the two out of three property and is closed under retracts.
\end{theorem}
\begin{proof}
Note that $U_1^\inv\calW_{\sCat}$ satisfies the two out of three property, so we already know that $f\in U_1^\inv\calW_{\sCat}$. 

It remains to show that $f\in \calL$.
Let $\dch = (c_1, \dots, c_n; d_1, \dots, d_m)$ be any biprofile in $\calQ$.
Since $g\in U_1^\inv\calW_{\sCat}$, the functor $\pi_0U_1(g) : \pi_0U_1 \calP \to \pi_0U_1 \calQ$ is essentially surjective. Thus, we can find a biprofile $(\uc'; \ud') = (c_1', \dots, c_n'; d_1', \dots, d_m')$ of $\calP$ along with isomorphisms 
\begin{align*}
	[a_i]=\alpha_i : g(c_i') &\to c_i \\
	[b_j]=\beta_j : d_j &\to g(d_j')
\end{align*}
in $\pi_0U_1 \calQ$.
We then have a diagram
\[ \begin{tikzcd} &
\calQ\dc \rar{f_{\dch}} \dar{(\underline a; \underline b)} &
\calR\fdc \dar{(\underline{fa}; \underline{fb})} 
\\
\calP \binom{\ud'}{\uc'} \rar{g_{(\uc'; \ud')}}
& \calQ\binom{g\ud'}{g\uc'} \rar{f_{(g\uc'; g\ud')}} & 
\calR\binom{fg\ud'}{fg\uc'}
\end{tikzcd} \]
The map $f_{(g\uc'; g\ud')}$ is a weak equivalence by two out of three on $\sSet$ and the vertical maps are weak equivalences by Corollary \ref{change of objects}.
Since the square commutes, $f_{\dch}$ is a weak equivalence as well. But $\dch$ was arbitrary, so $f\in \calL$.
\end{proof}

\subsection{Local liftings}

Our next goal is Proposition \ref{J liftings}, which  characterizes maps satisfying \eqref{F1} from Definition \ref{WEandfibrations} via a lifting property. 
We will also characterize maps which satisfy both \eqref{F1} and \eqref{W1}.

\begin{definition}\label{GNM characterization}
	For $n,m \geq 0$, let $\calG_{n,m} : \sSet \to \bN$ be the functor characterized by the property that
	\[
		\Hom_{\bN} (\mathcal{G}_{n,m}[X], \calP) = \left\{ \dch, f : |\uc| = n, |\ud| = m, f \in \Hom_{\sSet} \left(X, \calP \dc\right) \right\}.
	\]
\end{definition}

The following lemma says that these functors $\calG_{n,m}$ exist; its proof should be comfortable for any reader acquainted with the construction of $\Gamma(C_{(n;m)})$ from \cite[Definition 5.7]{hrybook}. When $\bN$ is the category of simplicial properads, these functors appeared previously in \cite{hackneyrobertson2}.

\begin{lemma}
Let $\Gr \in \{ \Gr^\uparrow, \Gr^\uparrow_{\textnormal{c}}, \Gr^\uparrow_{\textnormal{di}} \}$ and $n,m \geq 0$.
Then there is a functor $\calG_{n,m} : \sSet \to \sProp^{\Gr} = \bN$ as in Definition \ref{GNM characterization}. 
\end{lemma}

We require a bit of terminology that we will not use elsewhere in this paper, and which will be confined to the proof. 

\begin{proof}
Suppose that $\fC$ is a set, and let 
$\sProp^{\Gr(\fC)} \subseteq \bN$ denote the category with
\begin{itemize}
	\item objects those $\calP \in \bN$ with $\col(\calP) = \fC$, and
	\item morphisms those maps which are the identity on color sets.
\end{itemize}
There is a functor which picks out the underlying family of simplicial sets
\begin{align*}
	U : \sProp^{\Gr(\fC)} &\to \prod_{b,a \geq 0} \prod_{\fC^b \times \fC^a} \sSet
	\\
	\calP & \mapsto \left\{ \calP\dc \right\}.
\end{align*}
It admits a left adjoint $F$.

Given a function $f: \fC \to \fD$, there is a functor $f^* : \sProp^{\Gr(\fD)} \to \sProp^{\Gr(\fC)}$ which sends a simplicial $\Gr(\fD)$-prop $\calQ$ to a simplicial $\Gr(\fC)$-prop $f^*\calQ$ whose underlying family of simplicial sets is
\[
	U(f^*\calQ) = \left\{ f^* \calQ \dc = \calQ \fdc \right\}.
\]
Rephrasing our earlier definition, a morphism $\calP \to \calQ$ in $\bN$ is the same thing as a pair $(f : \fC \to \fD, \calP \to f^*\calQ)$ where $\fC = \col(\calP)$, $\fD = \col(\calQ)$, and $\calP \to f^*\calQ$ is a morphism in $\sProp^{\Gr(\fC)}$.

Now consider the set $\fE = \{ 1,2, \dots, n, 1', 2', \dots, m' \}$; we have that functions $\fE \to \fC$ are in bijection with biprofiles $\dc$ satisfying $|\uc| = n$ and $|\ud| = m$.
The projection $\prod_{b,a \geq 0} \prod_{\fE^b \times \fE^a} \sSet \to \sSet$ which picks out $\calP\binom{1', 2', \dots, m'}{1, 2, \dots, n}$ admits a left adjoint $F'$ (obtained by putting $\varnothing$ in all other entries).
Define $\calG_{n,m}[-]$ as the composite
\[
	\sSet \overset{F'}\to \prod_{b,a \geq 0} \prod_{\fE^b \times \fE^a} \sSet \overset{F}\to \sProp^{\Gr(\fE)} \to \sProp^{\Gr} = \bN.
\]

To see that this functor satisfies the desired universal property, we use that maps $\calG_{n,m}[X] \to \calP$ are the same thing as pairs $(f: \fE \to \fC, \calG_{n,m}[X] \to f^* \calP)$. As we said above, such an $f$ is the same thing as a biprofile $\dc = \binom{f1', \dots, fm'}{f1, \dots, fn}$ in $\fC$.
Further, since $\calG_{n,m}[X]$ is a free object in $\sProp^{\Gr(\fE)}$, maps $\calG_{n,m}[X] \to f^*\calP$ are in bijection with maps $X \to \calP \dc$ of simplicial sets.
\end{proof}

Recall that the Kan-Quillen model structure on $\sSet$ is cofibrantly generated, with generating cofibrations the boundary inclusions $\partial \Delta[p] \to \Delta[p]$ for $p\geq 0$ and generating acyclic cofibrations the horn inclusions $\Lambda[k,p] \to \Delta[p]$ with $0 \leq k \leq p$.

\begin{definition}\label{IJ sets}
Define two sets $I$ and $J$ of morphisms of $\bN$.
\begin{itemize}
\item The set $I$ consists of the maps 
\[ \calG_{n,m} [\partial \Delta[p] ] \to \calG_{n,m} [\Delta[p]],\]
where $n,m,p\geq 0$. 
\item The set $J$ consists of 
\[ \calG_{n,m} [\Lambda[k,p] ] \to \calG_{n,m} [\Delta[p]], \]where $n,m,p \geq 0$ and $0\leq k \leq p$. 
\end{itemize}
\end{definition}

\begin{proposition}\label{J liftings}
The class $\rlp J$ is the collection of all maps $f : \calP \to \calQ$ so that 
\[
	f_{\dch} : \calP \dc \to \calQ\fdc 
\]
is a fibration for all biprofiles $\dc$.
The class $\rlp I$ is the collection of all maps $f : \calP \to \calQ$ so that $f_{\dch} : \calP \dc \to \calQ\fdc$ is an acyclic fibration for all biprofiles $\dc$.
\end{proposition}
\begin{proof}
We prove the first statement, the second is analogous.
Suppose that $f : \calP \to \calQ$ is in $\rlp J$ and $\dc$ is a biprofile in $\col(\calP)$.
Suppose we have any diagram
\begin{equation} \label{diagram in sset}\begin{tikzcd}
\Lambda[k,p] \dar{i} \rar{g} & \calP \dc \dar{f_{\dch}} \\ \Delta[p] \rar{h} & \calQ\fdc;
\end{tikzcd} \end{equation}
then we have a commutative diagram
\begin{equation} \label{diagram in N}\begin{tikzcd}
\calG_{n,m} [\Lambda[k,p]] \dar[swap]{\calG_{n,m}[i]} \rar{\dch, g} & \calP  \dar{f} \\ \calG_{n,m} [\Delta[p]] \rar{\dcprime, h} & \calQ
\end{tikzcd} \end{equation}
where $\dcprime = \fdch$ 
by the defining property of $\calG_{n,m}$; since $f\in \rlp{J}$, the latter diagram has a lift, hence so does the former diagram. Thus $f_{\dch} \in \rlp{ (\Lambda[k,p]  \to \Delta[p])},$ the class of fibrations of $\sSet$.

Next suppose that $f : \calP \to \calQ$ is a map so that $f_{\dch}$ is a fibration for all $\dch$.
Then, given a diagram of the form \eqref{diagram in N}, 
the top map gives a biprofile $\dc$ in $\col(\calP)$ and the diagram \eqref{diagram in sset} commutes, where, necessarily, $\fdch = \dcprime$. But since $f_{\dch}$ is a fibration, a lift $t: \Delta[p] \to \calP\dc$ exists. This induces a lift in the diagram \eqref{diagram in N} by the universal property of Definition \ref{GNM characterization}. This is true for any diagram of this form, hence $f\in \rlp{J}$.
\end{proof}

By definition, $\calL$ consists of those maps which satisfy \eqref{W1} of Definition \ref{WEandfibrations}. The previous proposition establishes that $\rlp J$ consists of those maps which satisfy \eqref{F1} and $\rlp I$ consists of those maps that satisfy both \eqref{F1} and \eqref{W1}. We thus have the following corollary.

\begin{corollary}\label{local liftings}
Let $\bN = \sProp^{\Gr}$ be the category of simplicial props, simplicial properads, or simplicial dioperads.
If $I$ and $J$ are as in Definition \ref{IJ sets}, then
\[ \rlp I = \calL \cap \rlp J. \] \qed
\end{corollary}

This verifies \eqref{ST6} of Theorem \ref{squeezy theorem} in the proof of the main theorem (using $\Gr = \Gr^\uparrow_\text{c}$ in \ref{properadstheorem} and $\Gr = \Gr^\uparrow_\text{di}$ in \ref{dioperadstheorem}).

\section{Adjunctions over \texorpdfstring{$\sCat$}{sCat}}\label{section adj}

Suppose that we have an adjunction $(F_2, U_2)$ over $\sCat$:
\[ \begin{tikzcd}
\sCat \dar[swap]{\pi_0} \arrow[r, shift left, "F_1"] \arrow[rr, bend left=50, swap,  "F_0"]& 
\bN \arrow[r, shift left, "F_2"] \arrow[l, shift left, "U_1"] &
\bP \arrow[l, shift left, "U_2"] \arrow[ll, bend left=50, swap,  "U_0"] \\
\Cat
\end{tikzcd} \]
Write $\eta: \id_{\bN} \Rightarrow U_2F_2$ for the unit of the adjunction.
We say that $(F_2, U_2)$ is \emph{categorically well-behaved} if 
\begin{enumerate}
	\item $\pi_0 U_1 (\eta_X) : \pi_0 U_1 X \to \pi_0 U_1 U_2F_2 X$ is the identity on objects for all $X$. \label{CATWB1}
	\item The map
	\[
		\Iso(\pi_0 U_1 X) \to \Iso(\pi_0 U_1 U_2F_2 X)
	\]
	induced by $\pi_0 U_1 \eta_X$
	is a bijection for all $X$. \label{CATWB2}
\end{enumerate}

\begin{proposition}\label{categorically well-behaved prop}
Suppose that the adjunction is categorically well-behaved and $f : X \to Y \in \bN$.
	\begin{itemize}
		\item If $\pi_0U_1U_2F_2(f)$ is essentially surjective, then so is $\pi_0U_1(f)$. 
		\item If $\pi_0U_1U_2F_2(f)$ is an isofibration, then so is $\pi_0U_1(f)$. 
	\end{itemize}
\end{proposition}
\begin{proof}
Within this proof, we will write $T_2 = U_2F_2$.

	Suppose that $\pi_0U_1T_2(f)$ is essentially surjective and $b\in \pi_0U_1Y$. Then there is an isomorphism
	$\phi : \pi_0U_1T_2(f)(a) \to b$ in $\pi_0U_1T_2Y$.
	By \eqref{CATWB1}, $\pi_0U_1T_2(f)(a) = \pi_0U_1(f)(a)$ and by \eqref{CATWB2} this $\phi$ comes from an isomorphism $\tilde \phi : \pi_0U_1(f)(a) \to b$ in $\pi_0U_1Y$. Since $b$ was arbitrary, $\pi_0U_1(f)$ is essentially surjective.

	Now suppose that $\pi_0U_1T_2(f)$ is an isofibration and 
	\[
		\tilde \phi : \pi_0U_1(f)(e) \to b
	\]
	is an isomorphism in $\pi_0U_1Y$. Then $\phi = \pi_0 U_1 (\eta_Y)(\tilde \phi)$ is an isomorphism in $\pi_0U_1T_2Y$, hence by the isofibration property there is a lift
	\[ \begin{tikzcd}[row sep=small]
	e \rar{\psi} & e' \\
	\pi_0U_1T_2(f)(e) \rar{\phi} & b
	\end{tikzcd} \]
	with $\psi: e \to e'$ an isomorphism of $\pi_0U_1T_2X$ and $\pi_0U_1T_2(f)(\psi) = \phi$.
	By \eqref{CATWB2}, there is a $\tilde \psi : e\to e'$ in $\pi_0U_1X$ so that $\pi_0 U_1 (\eta_X)(\tilde \psi) = \psi$.
	But then 
	\[ \pi_0 U_1 (\eta_Y)(\tilde \phi) = \phi = \pi_0U_1T_2(f)(\psi) = \pi_0U_1T_2(f)(\pi_0 U_1 (\eta_X)(\tilde \psi)) = \pi_0 U_1 (\eta_Y)(\pi_0U_1(f)(\tilde \psi)) \]
	so by injectivity of $\pi_0 U_1 (\eta_Y)$ on isomorphisms, we must have
	$\pi_0U_1(f)(\tilde \psi) = \tilde \phi$. 
	Since $\tilde \phi$ was arbitrary, we then have that $\pi_0U_1(f)$ is an isofibration.
\end{proof}

\subsection{Left adjoints}\label{left adjoints}

We have a sequence of adjunctions
\[ \begin{tikzcd}
\sCat
\arrow[r, shift left]
&
\sDioperad 
\arrow[r, shift left, "F^1"] \arrow[l, shift left]
&
\sProperad
\arrow[r, shift left, "F^0"] \arrow[l, shift left, "U^1"]
&
\sProp
\arrow[l, shift left, "U^0"]
\end{tikzcd} \]
Our next goal is to show that the adjunctions $F^1 \colon \sDioperad \rightleftarrows \sProperad \colon U^1$ and $F^0 \colon \sProperad \rightleftarrows \sProp \colon U^0$ over $\sCat$ are categorically well-behaved.
To do so, we recall the description of the left adjoints from \cite[\S 12.1.3]{yj}.

\newcommand{\extensioncat}{\calD_{\fC}\dc}
\newcommand{\extensioncategoryone}{\calD_{\fC}^1\dc}
\newcommand{\extensioncategoryzero}{\calD_{\fC}^0\dc}
\newcommand{\extensioncati}{\calD_{\fC}^i\dc}

Let $\Gr \leq \Gr'$ be a pair of pasting schemes, $\fC$ a set of colors, and $\dch$ be a biprofile in $\fC$.
The \emph{extension category} $\extensioncat$ has objects $\Ob(\extensioncat) = \Gr'(\fC)\dc$.
A morphism\footnote{To ensure that $\Hom(K,G)$ is a \emph{set} instead of a proper class, one should make an assumption that the substitution data is given, as in Remark \ref{remark strict iso}, by strict isomorphism classes of graphs.} $K \to G$ consists of substitution data $\{ H_v \}$ so that $K = G\{H_v\}$ where $H_v \in \Gr\profilev$.
Composition is given by associativity of graph substitution, that is, given $\{ I_w \} : J \to K$ and $\{ H_v \} : K \to G$, use the isomorphism $\vertex(K) \cong \coprod_{v\in G} \vertex(H_v)$ to reindex, so $\{I_w\}_{w\in K} = \{I_u^v\}_{v\in G, u\in H_v}$;  we then have
\[
	J = K\{I_w\}_{w\in K} =  (G\{H_v\}_{v\in G})\{I_w\} = G\{ H_v \{I_u^v \}_{u\in H_v} \}_{v\in G}. 
\]
This gives a morphism $\{ H_v \{I_u^v \} \} : J \to G$ since each $H_v \{I_u^v \} \in \Gr\profilev$.
Recall from \cite[Lemma 12.6]{yj} that the entries of the functor
\[
	F : \sProp^{\Gr} \to \sProp^{\Gr'}
\]
are given by
\begin{equation}\label{general left adjoint}
	F\calP \dc = \colim_{\extensioncat} \calP[G].
\end{equation}
The unit of the adjunction is given at a biprofile $\dc$ by the inclusion $\{ C_{\dch} \} \hookrightarrow \extensioncat$ 
yielding
\begin{equation}\label{unit of adjunction thing}
	\calP \dc = \calP[C_{\dch}] \to \colim_{\extensioncat} \calP[G] = F\calP \dc.
\end{equation}

In the case when $\Gr \leq \Gr'$ is the pair of pasting schemes 
\[
	\Gr^\uparrow_\text{c} \leq \Gr^\uparrow
\]
we will write the extension category as $\extensioncategoryzero$ and when we have
the pair of pasting schemes
\[
	\Gr^\uparrow_\text{di} \leq \Gr^\uparrow_\text{c}
\]
we will write the extension category as $\extensioncategoryone$.
Thus, for $i=0,1$ we have
\[ 
	F^i\calP \dc = \colim_{\extensioncati} \calP[G].
\]

We know that
\[
	\Ob\left(\extensioncati\right) = \begin{cases}
		\Gr^\uparrow(\fC)\dc & i=0 \\
		\Gr^\uparrow_\text{c}(\fC)\dc & i=1.
	\end{cases}
\]
By forgetting structure, each graph $G$ is a (1-skeletal) CW complex, so we can define a map $\beta_i$ (for `Betti number')
\begin{align*}
\beta_i \colon \Ob\left(\extensioncati\right)  &\to \mathbb N \\
G & \mapsto \operatorname{rank} \tilde H_i(G; \mathbb Z)
\end{align*}
Suppose that $G\{H_v \} \to G$ is a morphism of $\extensioncategoryzero$.
Since each $H_v$ is connected, both $G\{H_v\}$ and $G$ have the same number of connected components, hence $\beta_0$ extends to a functor from $\extensioncategoryzero$ to the discrete category $\mathbb N$.
If $G\{H_v \} \to G$ is a morphism of $\extensioncategoryone$ then each $H_v$ is in $\Gr^\uparrow_\text{di}$, hence contractible, so $\beta_1 (G\{H_v \}) = \beta_1(G)$.
Thus $\beta_1$ extends to a functor from $\extensioncategoryone$ to the discrete category $\mathbb N$.

We have thus shown that the extension categories split, that is, that 
\[
	\extensioncati = \coprod_{j \geq 0} \beta_i^{-1}(j).
\]
This implies that the colimits split, so we have
\begin{equation}\label{functorial F splitting}
	F^i\calP \dc = \colim_{\extensioncati} \calP[G] = \coprod_{j \geq 0} \colim_{\beta_i^{-1}(j)} \calP[G]
\end{equation}
and this splitting is respected by the maps $(F^if)_{\dch} : F^i\calP \dc \to F^i\calQ \fdc$ for any $f: \calP \to \calQ$.

Suppose that $G$ is any graph.
We then have $\beta_0(G) = 0$ if and only if $G\in \Gr^\uparrow_\text{c}$. In this case, we also have $\beta_1(G) = 0$ if and only if $G\in \Gr^\uparrow_\text{di}$.

\begin{proposition}\label{betazerodecomp}
For $i=0,1$, there is a splitting 
\[
	F^i\calP\dc = \calP\dc \amalg \widetilde F^i_{\dch}(\calP),
\]
functorial in maps $\calP \to \calQ$ of simplicial properads (for $i=0$) or simplicial dioperads (for $i=1$).
\end{proposition}
\begin{proof} The splitting comes from \eqref{functorial F splitting}. We already mentioned that this splitting extends to maps.
Set 
\[
	\widetilde F^i_{\dch} (\calP) = \coprod_{j \geq 1} \colim_{\beta_i^{-1}(j)} \calP[G] = \colim_{\beta_i^{-1}[1,\infty)} \calP[G].
\]
The subcategory $\beta_i^{-1}(0) \subset \extensioncati$ contains a terminal object $C_{\dch}$ (cf.  \cite[2.12]{hackneyrobertson2}), so \[
\colim_{\beta_i^{-1}(0)} \calP[G] = \calP[C_{\dch}] = \calP \dc. \]
\end{proof}

This proposition and its proof shows that the unit of the adjunction \eqref{unit of adjunction thing} is injective, so we have the following corollary.

\begin{corollary}\label{Lifaithful}
	The functor $F^i$ is faithful. \qed
\end{corollary}

\begin{proposition}
	The adjunctions 
\[ \begin{tikzcd}
\sCat
\arrow[r, shift left]
&
\sDioperad 
\arrow[r, shift left, "F^1"] \arrow[l, shift left]
&
\sProperad 
\arrow[l, shift left, "U^1"]
\end{tikzcd}\]
and
\[ 
\begin{tikzcd}
\sCat
\arrow[r, shift left]
&
\sProperad
\arrow[r, shift left, "F^0"] \arrow[l, shift left]
&
\sProp
\arrow[l, shift left, "U^0"]
\end{tikzcd}
\]
over $\sCat$ are categorically well-behaved.
\end{proposition}
\begin{proof}
For uniformity of argument, we will generically write
\[
	F \colon \sCat \rightleftarrows \sDioperad \colon U \qquad \& \qquad F \colon \sCat \rightleftarrows \sProperad \colon U
\] for the adjunctions to $\sCat$ and 
\[ \Gr\dc = \begin{cases} 
\Gr^\uparrow\dc & i=0 \\
\Gr^\uparrow_\text{c}\dc &  i=1. \end{cases} \]

The first condition is automatic since $\calP$ and $U^iF^i \calP$ have the same set of colors, which gives the object set for $\pi_0U\calP$ and $\pi_0UU^iF^i\calP$.

Since $F^i$ is faithful by Corollary \ref{Lifaithful}, we know that the map on isomorphism sets is injective.
Suppose that we have an isomorphism
\[
	\alpha \in \pi_0 (U U^iF^i \calP) (x,y).
\]
We wish to show that $\alpha$ was actually already in $\pi_0 (U \calP) (x,y)$.
Then $\alpha$ is represented by the image $\bar a$ of some vertex
\[
	a \in \calP [G]_0 \to F^i\calP \binom{y}{x}_0
\]
for some $G \in \Gr\binom{y}{x}$.
Let $a' \in \calP[G']_0 \to F^i\calP \binom{x}{y}_0$ be a vertex whose image $\bar a'$ represents $\alpha^{-1} \in \pi_0 (UU^iF^i \calP) (y,x)$.
Consider the graph $G' \circ_1 G \in \Gr\binom{x}{x}$,
\begin{center}
\includegraphics[height=3.5cm]{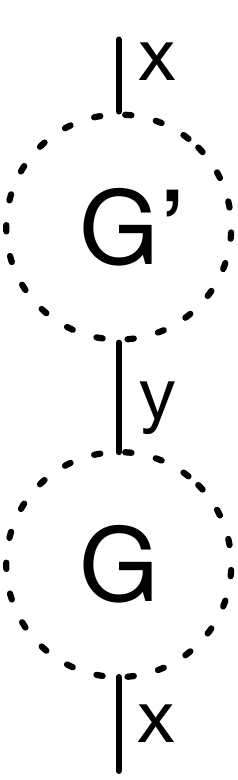}
\end{center}
given by grafting the output edge of $G$ to the input edge of $G'$.
We know that $\bar a' \circ_1 \bar a$ can be obtained by looking at $a' \circ_1 a \in \calP[G'\circ_1 G]$; we then have 
\[ \begin{tikzcd}[row sep=tiny]
 \calP[G'\circ_1 G] \rar & F^i\calP \binom{x}{x} & \calP[C_{(x;x)}] \lar \\
 a' \circ_1 a \rar[mapsto] & \bar a' \circ_1 \bar a \sim \id_x & \id_x \lar[mapsto] 
\end{tikzcd} \]
Since $\bar a' \circ_1 \bar a$ and $\id_x$ represent the same element of $\pi_0 F^i\calP \binom{x}{x}$, we have $\beta_i (G' \circ_1 \bar G) = \beta_i C_{(x;x)} = 0$.
But by Mayer-Vietoris (for reduced homology), $\beta_i(G') + \beta_i(G) = \beta_i(G'\circ_1 G)$, hence $\beta_i(G) = 0$ as well.
Thus the image of $a$ in $F^i\calP\binom{y}{x}$ is represented by an element $a''\in \calP\binom{y}{x}_0$ by Proposition \ref{betazerodecomp}.
It follows that $\alpha = [a] = [a''] \in \pi_0 (U \calP) (x,y)$, as we wished to show.
\end{proof}

\section{The model structure on simplicial properads}\label{section model}

In Definition \ref{IJ sets}, we gave two sets of maps of $\sProp^{\Gr}$. We now give two sets of maps of $\sCat$.
If $X$ is a simplicial set, write $\calG_{1,1}[X]$ for the simplicial category with two objects $x,y$ and
\begin{align*}
\Hom(x,x) &= \Delta[0] & \Hom(x,y) &= X \\
\Hom(y,y) &= \Delta[0] & \Hom(y,x) &=\varnothing.
\end{align*}
As in Definition \ref{GNM characterization}, we consider $\calG_{1,1}[-]$ as a functor from $\sSet$ to $\sCat$.
Let $\calI$ be the category with one object $x$ and no non-identity morphisms.
We consider the class of simplicial categories $\calH$ with two objects $x$ and $y$, weakly contractible function complexes, and only countably many simplices in each function complex. Furthermore, we require that each such $\calH$ is cofibrant in the Dwyer-Kan model category structure on $\sCat_{\{x,y\}}$ \cite[7.1.(iii)]{dk-simploc}. Let $\mathbf{H}$ denote a set of representatives of isomorphism classes of such categories.

\begin{definition}\label{generatingcof}
The set $I_0$ consists of the following simplicial functors:
\begin{enumerate}[label={(C{\arabic*})},ref={C\arabic*}]
        \item\label{C1} For $p\geq 0$, the maps $\calG_{1,1} [\partial \Delta[p] ] \to \calG_{1,1} [\Delta[p]]$.
        \item\label{C2} The $\sSet$-functor $\varnothing\hookrightarrow \calI$.
\end{enumerate}
The set $J_0$ of consists of the following simplicial functors:
\begin{enumerate}[label={(A{\arabic*})},ref={A\arabic*}]
        \item\label{A1} For $p \geq 0$ and $0\leq k \leq p$, the maps $\calG_{1,1} [\Lambda[k,p] ] \to \calG_{1,1} [\Delta[p]]$.
        \item\label{A2} The $\sSet$-functors $\calI\hookrightarrow\calH$ 
        for $\calH \in \mathbf{H}$ which take $x$ to $x$.
\end{enumerate}
\end{definition}

Note that \eqref{C1} and \eqref{A1} give non-empty intersections $F_1(I_0) \cap I$ and $F_1(J_0) \cap J$.

\begin{theorem}[Characterization of fibrations]
	A map $f \in \sProp^{\Gr}$ is a fibration in the sense of Definition \ref{WEandfibrations} if and only if $f\in \rlp{(F_1 J_0 \cup J)}$.
\end{theorem}
\begin{proof}
By Proposition \ref{J liftings}, $f\in \rlp{J}$ if and only if $f$ satisfies \eqref{F1}. 
By Lemma \ref{rlp adjunction}, $\rlp{(F_1J_0)} = U_1^\inv (\rlp{J_0})$, and $\rlp{J_0}$ is the class of fibrations in $\sCat$.
Thus $f : \calP \to \calQ \in \rlp{(F_1J_0)}$ if and only if $U_1(f)$ is a fibration in $\sCat$ if and only if $\pi_0(f)$ is an isofibration and $f : \calP \binom{d}{c} \to \calQ \binom{fd}{fc}$ is a fibration for all $c,d$.

Thus maps in $\rlp{(F_1 J_0 \cup J)} = \rlp{(F_1J_0)} \cap \rlp{J}$ satisfy both \eqref{F1} and \eqref{F2}. Moreover, if $f$ satisfies \eqref{F1} and \eqref{F2}, then $U_1(f)$ is a fibration in $\sCat$ and $f\in \rlp{J}$, hence $f\in \rlp{(F_1J_0)} \cap \rlp{J}$.
\end{proof}

\begin{theorem}\label{properadstheorem}
The category $\sProperad$ of all small properads admits a (cofibrantly-generated) model structure with fibrations and weak equivalences as in Definition \ref{WEandfibrations}.
The set $F_1I_0 \cup I$ is a set of generating cofibrations and the set $F_1J_0\cup J$ is a set of generating acyclic cofibrations.
\end{theorem}
\begin{proof}
We will apply Proposition \ref{squeezy theorem}, using the adjunctions\footnote{The adjoint pair $(F_2,U_2)$ was called $(F^0,U^0)$ in the previous section.}
\[ \begin{tikzcd}
 \bM = \sCat \arrow[r, shift left, "F_1"] & 
\sProperad \arrow[r, shift left, "F_2"] \arrow[l, shift left, "U_1"] &
\sProp. \arrow[l, shift left, "U_2"] 
\end{tikzcd} \]
In this situation, $\calW_{\bP} = U_0^\inv \calW_{\bM} \cap U_2^\inv \calL$.

First note that $\sProperad$ is complete and cocomplete.
Clearly $\varnothing$ and $\calI$ are finite.
By the characterization of Definition \ref{GNM characterization} and a variation on \cite[3.1.2]{hovey}, we also have that 
$\calG_{n,m} [\partial \Delta[p] ]$ and $\calG_{n,m} [\Lambda[k,p] ]$ are finite.
This implies that all of these objects are small relative to both $(F_0I_0\cup I)\text{-cell}$ and $(F_0J_0\cup J)\text{-cell}$, so \eqref{ST3} and \eqref{ST4} both hold.
\eqref{ST5} is established by Theorem \ref{lemma twoofthree}.
Corollary \ref{local liftings} ensures that \eqref{ST6} holds.

Consider the class $\calL \subset \sProperad$ of local equivalences.
To show that $F_2^{-1}(U_2^{-1}\calL) \subset \calL$, suppose that $f : \calP \to \calQ$ is in the former class. 
Using Proposition \ref{betazerodecomp} we then have a diagram
\[ \begin{tikzcd}
\calP \dc \dar[hook] \arrow[rrr, "f"] & & & \calQ \fdc \dar[hook] \\
\calP \dc \amalg X \rar[equals] & F_2\calP\dc \rar{F_2(f)} & F_2\calQ\fdc & \calQ \fdc \amalg Y. \lar[equals]
\end{tikzcd} \]
By Proposition \ref{betazerodecomp}, $F_2(f)$ is a coproduct $f \amalg (X\to Y)$.
By assumption $F_2(f)$ is a weak equivalence in $\sSet$, hence so is $f$.
This shows
\begin{equation}\label{first equation}
F_2^{-1}(U_2^{-1}\calL) \subset \calL.
\end{equation}

Suppose that $F_2(f) \in \calW_{\bP} = (U_0^\inv\wm) \cap (U_2^{-1}\calL)$. 
Since $F_2(f) \in U_0^\inv\wm$, we know that $\pi_0U_0F_2(f)$ is essentially surjective, hence so is  
that $\pi_0U_1(f)$ by Proposition \ref{categorically well-behaved prop}.
Since $F_2(f) \in U_2^{-1}\calL$, we know (from the previous paragraph) that $f\in \calL$, hence $U_1(f)$ is levelwise an equivalence. Thus $\pi_0U_1(f)$ is fully-faithful so $\pi_0U_1(f)$ is an equivalence of categories and we have
\begin{equation}\label{second equation}
F_2^{-1} ( \calW_{\bP})  = F_2^{-1} \Big( (U_0^\inv\wm) \cap (U_2^{-1}\calL) \Big) \subset U_1^\inv(\wm).
\end{equation}

Combining \eqref{first equation} and \eqref{second equation}, we then have 
\[
	F_2^{-1} ( \calW_{\bP})  = F_2^{-1} \Big( (U_0^\inv\wm) \cap (U_2^{-1}\calL) \Big) \subset U_1^\inv(\wm) \cap \calL = \calW_{\bN},
\]
namely that \eqref{ST7} holds.

The fact that \eqref{ST1} holds for $\bM = \sCat$ is the main theorem of \cite{bergner}. In the case of the adjunction $\bN = \sProperad \rightleftarrows \sProp = \bP$, we have that \eqref{ST2} holds by the main theorem of \cite{hackneyrobertson2}.
Then apply Proposition \ref{squeezy theorem} to get the appropriate model structure on $\sProperad$.
\end{proof}

\begin{theorem}\label{dioperadstheorem}
The category $\sDioperad$ admits a (cofibrantly-generated) model structure with fibrations and weak equivalences as in Definition \ref{WEandfibrations}.
\end{theorem}
\begin{proof}
As in the proof of the previous theorem, one applies Proposition \ref{squeezy theorem}, this time using the adjunction $\bN = \sDioperad \rightleftarrows \sProperad = \bP$.
The only change necessary (other than changing `$\sProperad$' to `$\sDioperad$') is that \eqref{ST2} for $\bP = \sProperad$ is Theorem \ref{properadstheorem}.
\end{proof}

\begin{remark}
The method used in the previous two theorems does not apply to get a model structure on the category $\bN = \mathbf{sOperad}$ of simplicially-enriched operads. We cannot apply \ref{squeezy theorem} with $\bP = \sProp, \sProperad,$ or $\sDioperad$, since \eqref{ST2} will not hold for such $\bP$: operads have far fewer underlying entries than props, properads, and dioperads.
\end{remark}

\addcontentsline{toc}{section}{References}

\begin{thebibliography}{10}

\bibitem{bergner}
Julia~E. Bergner.
\newblock A model category structure on the category of simplicial categories.
\newblock {\em Trans. Amer. Math. Soc.}, 359(5):2043--2058, 2007.

\bibitem{caviglia}
Giovanni Caviglia.
\newblock The {D}wyer-{K}an model structure for enriched coloured
  {P}{R}{O}{P}s.
\newblock Preprint, \href{http://arxiv.org/abs/1510.01289}{arXiv:1510.01289} [math.CT].

\bibitem{cm-ho}
Denis-Charles Cisinski and Ieke Moerdijk.
\newblock Dendroidal sets as models for homotopy operads.
\newblock {\em J. Topol.}, 4(2):257--299, 2011.

\bibitem{cm-ds}
\bysame.
\newblock Dendroidal {S}egal spaces and $\infty$-operads.
\newblock {\em J. Topol.}, 6(3):675--704, 2013.

\bibitem{cm-simpop}
\bysame.
\newblock Dendroidal sets and simplicial operads.
\newblock {\em J. Topol.}, 6(3):705--756, 2013.

\bibitem{dk-simploc}
W.~G. Dwyer and D.~M. Kan.
\newblock Simplicial localizations of categories.
\newblock {\em J. Pure Appl. Algebra}, 17(3):267--284, 1980.

\bibitem{gan}
Wee~Liang Gan.
\newblock Koszul duality for dioperads.
\newblock {\em Mathematical Research Letters}, 10(1):109, 2003.

\bibitem{MR2414322}
Richard Garner.
\newblock Polycategories via pseudo-distributive laws.
\newblock {\em Adv. Math.}, 218(3):781--827, 2008.

\bibitem{hackneyrobertson2}
Philip Hackney and Marcy Robertson.
\newblock The homotopy theory of simplicial props.
\newblock {\em Israel J. Math.}, 219(2):835--902, 2017.

\bibitem{hry3}
Philip Hackney, Marcy Robertson, and Donald Yau.
\newblock The homotopy coherent nerve for properads.
\newblock In preparation.

\bibitem{hrybook}
\bysame.
\newblock {\em Infinity {P}roperads and {I}nfinity {W}heeled {P}roperads}.
\newblock Lecture Notes in Mathematics, Vol. 2147. Springer, 2015.

\bibitem{hovey}
Mark Hovey.
\newblock {\em Model categories}, volume~63 of {\em Mathematical Surveys and
  Monographs}.
\newblock American Mathematical Society, Providence, RI, 1999.

\bibitem{MR1888677}
Michele Intermont and Mark~W. Johnson.
\newblock Model structures on the category of ex-spaces.
\newblock {\em Topology Appl.}, 119(3):325--353, 2002.

\bibitem{jy}
Mark~W. Johnson and Donald Yau.
\newblock On homotopy invariance for algebras over colored {P}{R}{O}{P}s.
\newblock {\em J. Homotopy Relat. Struct.}, 4(1):275--315, 2009.

\bibitem{joyal_rigid}
A.~Joyal.
\newblock Simplicial categories vs quasi-categories.
\newblock Unpublished manuscript, 2006.

\bibitem{joyal}
\bysame.
\newblock Quasi-categories and {K}an complexes.
\newblock {\em J. Pure Appl. Algebra}, 175(1-3):207--222, 2002.
\newblock Special volume celebrating the 70th birthday of Professor Max Kelly.

\bibitem{htt}
Jacob Lurie.
\newblock {\em Higher topos theory}, volume 170 of {\em Annals of Mathematics
  Studies}.
\newblock Princeton University Press, Princeton, NJ, 2009.
\newblock Available at
  \url{http://www.math.harvard.edu/~lurie/papers/highertopoi.pdf}.

\bibitem{catalg}
Saunders Mac~Lane.
\newblock Categorical algebra.
\newblock {\em Bull. Amer. Math. Soc.}, 71:40--106, 1965.

\bibitem{MR2397628}
S.~A. Merkulov.
\newblock Lectures on {PROP}s, {P}oisson geometry and deformation quantization.
\newblock In {\em Poisson geometry in mathematics and physics}, volume 450 of
  {\em Contemp. Math.}, pages 223--257. Amer. Math. Soc., Providence, RI, 2008.

\bibitem{mw}
Ieke Moerdijk and Ittay Weiss.
\newblock Dendroidal sets.
\newblock {\em Algebr. Geom. Topol.}, 7:1441--1470, 2007.

\bibitem{mw2}
\bysame.
\newblock On inner {K}an complexes in the category of dendroidal sets.
\newblock {\em Adv. Math.}, 221(2):343--389, 2009.

\bibitem{riehl}
Emily Riehl.
\newblock {\em Categorical homotopy theory}, volume~24 of {\em New Mathematical
  Monographs}.
\newblock Cambridge University Press, Cambridge, 2014.

\bibitem{MR0373846}
M.~E. Szabo.
\newblock Polycategories.
\newblock {\em Comm. Algebra}, 3(8):663--689, 1975.

\bibitem{vallette}
Bruno Vallette.
\newblock A {K}oszul duality for {PROP}s.
\newblock {\em Trans. Amer. Math. Soc.}, 359(10):4865--4943, 2007.

\bibitem{yj}
Donald Yau and Mark~W. Johnson.
\newblock {\em A {F}oundation for {P}{R}{O}{P}s, {A}lgebras, and {M}odules}.
\newblock Math. Surveys and Monographs, Vol. 203. Amer. Math. Soc., Providence,
  RI, 2015.

\end{thebibliography}
\providecommand{\bysame}{\leavevmode\hbox to3em{\hrulefill}\thinspace}

\end{document}